\documentclass[reqno, a4paper, 11pt]{amsart}
\usepackage{amssymb,url, color, mathrsfs}
\usepackage[colorlinks=true, bookmarks=true, pdfstartview=FitH, pagebackref=true, linktocpage=true]{hyperref}
\usepackage[short,nodayofweek]{datetime}
\usepackage[pagewise,running,mathlines,displaymath, switch]{lineno}
\usepackage{times}
\usepackage{indentfirst, tabularx}
\usepackage[table]{xcolor}
\usepackage{float}

\usepackage[T5]{fontenc}
\usepackage[utf8]{inputenc}

\theoremstyle{plain}
\numberwithin{equation}{section}
\newtheorem{theorem}{Theorem}[section]
\newtheorem{lemma}{Lemma}[section]

\newtheorem{proposition}{Proposition}[section]
\theoremstyle{remark}

\definecolor{brown}{rgb}{0.5,0,0}
\definecolor{backgroundcolor}{rgb}{0.98, 0.92, 0.73}

\def\N{\mathbb N}
\def\R{\mathbf R}

\def\S{\mathbb S}

\begin{document}
\allowdisplaybreaks

\setpagewiselinenumbers
\setlength\linenumbersep{100pt}

\title[Positive entire solutions of $\Delta^2 u+u^{-q}=0$ in $\R^3$ ]
{A necessary and sufficient condition for radial property of positive entire solutions of $\Delta^2 u+u^{-q}=0$ in $\R^3$}

\begin{abstract}
In this article, we are concerned with the following geometric equation 
\begin{equation}\label{MainEq}
\Delta^2 u = -u^{-q} \qquad \text{in } \R^3
\end{equation}
for $q>0$.
Recently,  Guo, Wei and Zhou  have established the relationship between the radial symmetry and the exact growth rate at infinity of a positive entire solution of  that equation as $1<q<3$. The aim of this paper is to obtain the similar result in the case $q>3$ thanks to the method of moving plane. 
\end{abstract}

\date{\bf \today \; at \, \currenttime}

\subjclass[2010]{34A34, 35A15, 35B06, 35B40, 35C20}

\keywords{Biharmonic equation, radial solution, method of moving plane}

\author[N. T. Tài ]{Nguyễn Tiến Tài}
\address[N. T.  Tài ]{Laboratoire Analyse G\'eom\'etrie et Applications, Universit\'e Sorbone Paris Nord,  93430 - Villetaneuse, France}
\email{\href{mailto: N.T. Tài <tientai.nguyen@math.univ-paris13.fr>}{tientai.nguyen@math.univ-paris13.fr}}

\maketitle

\section{Introduction}

In this article, we are interested in studying a necessary and sufficient condition for positive entire solutions to be radially symmetric of the following geometric equation 
\begin{equation}\label{MainEq}
\Delta^2 u = -u^{-q} \qquad \text{in } \R^3,
\end{equation}
provided that $q>0$.  Eq. \eqref{MainEq} has attracted many mathematicians over years  since its root from the prescribed $Q$-curvature problem in conformal geometry. We refer  interested readers to \cite{CX09} for further  information. 

The existence of a positive entire solution of \eqref{MainEq} was first proved in \cite{CX09} that there holds $q>1$, necessarily. As $q>1$,  McKenna and Reichel   \cite{KR03} looked for the radial solutions  of Eq. \eqref{MainEq} via shooting method. 
To be precise, Eq. \eqref{MainEq} was transformed into the following initial value problem
\begin{equation}\label{OdeTransform}
\begin{cases}
\Delta^2 u = -u^{-q}, \quad r \in (0,R_{\max}(\beta)),\\
 u(0)=1, \quad u'(0)=0,\\
 \Delta u(0) =\beta>0, \quad (\Delta u)'(0)=0.
\end{cases}
\end{equation}
Here $R_{\max}(\beta)$ is the largest radius of the interval of existence of the solution. The result in \cite{KR03} asserted that there exists a unique threshold parameter $\beta^{\star}$ such that $R_{\max}(\beta)=\infty$ if $\beta \geqslant \beta^{\star}$.  The authors also claimed that $u_{\beta} > u_{\beta^{\star}}$ in $(0,\infty)$ for $\beta>\beta^{\star}$ thanks to a comparison principle stated in \cite[Lemma 3.2]{KR03}. Then, we say that  $u_{\beta^{\star}}$ is the (unique) minimal  entire radial solution of \eqref{MainEq} and $ (u_{\beta})_{\beta>\beta^{\star}}$ are a family of  non--minimal radial solutions of \eqref{MainEq}. Moreover, the asymptotic behavior of  all radially symmetric solutions of \eqref{MainEq} was classified in \cite{DFG10,Gue12,DN17}. A complete picture of radial entire solutions of Eq. \eqref{MainEq} was demonstrated in \cite[Table 1.1]{DN17}. To be precise, 
\begin{equation}\label{RateMinimal}
u_{\beta^{\star}}(r) = 
\begin{cases}
O( r^{\frac4{1+q}}) \quad&\text{if } 1<q<3,\\
 O(r\log r )\quad&\text{if } q=3,\\
O( r) \quad&\text{if } q>3,
\end{cases}
 \end{equation}
 and 
 \begin{equation}\label{RateNonMinimal}
 u_{\beta}(r) =O(r^2) \quad\text{if } q>1 \text{ and } \beta >\beta^{\star}.
 \end{equation}

Having \eqref{RateMinimal} and \eqref{RateNonMinimal} at hand,  Guo, Wei and Zhou study further  the radial property  of singular positive entire solutions of \eqref{MainEq} in \cite{GWZ18}. The authors show the necessary and sufficient condition to claim that a positive entire solution $u(x)$ of \eqref{MainEq} as $1<q<3$ which grows like a  minimal  radial entire solution of \eqref{MainEq} at infinity is actually a minimal radial entire solution of \eqref{MainEq}.
And that goal was also achieved for a positive regular entire solution $u$ of \eqref{MainEq} which has  the asymptotic behavior at infinity as that of a non--minimal entire radial solution as $q>1$.

In view of Guo-Wei-Zhou's result, we answer the same question for  a positive singular entire solution $u$ of \eqref{MainEq} which admits the asymptotic behavior at infinity as that of a minimal radial entire solution as $q>3$.
  The main theorem is as follows.
 \begin{theorem}\label{MainThm}
Let $q>3$ and $u \in C^4(\R^3)$ be a positive entire solution of \eqref{MainEq}.
Then, $u$ is a minimal radial entire solution of \eqref{MainEq} if and only if  there exists  $0< \vartheta < 1$  and $L>0$ such that
\begin{equation}\label{MainCondition}
|x|^{-1}u(x) - L = o(|x|^{-\vartheta})
\end{equation}
as $|x| \to \infty$.
\end{theorem}

Inspired by the previous articles \cite{Zou95,Guo02,GW07, GHZ15, GW18}, where the authors studied that problem on equations of the form \eqref{MainEq} involving Laplacian and bi--Laplacian also, the method of moving plane is still our key ingredient  to show the necessary part of the main theorem. We first collect in Section \ref{SectPreparation} some basic properties of the eigenvalues and eigenfunctions of Laplacian and bi--Laplacian on $\S^2$ and  introduce the Kelvin transform of solution $u$, i.e. 
\begin{equation}\label{KelvinTransform}
y=\frac{x}{r^2},\quad r=|x|>0,\quad v(y) = |x|^{-1} u(x) - L.
\end{equation}
Section \ref{SectProperties} is devoted to establish an upper bound for  
\begin{equation}\label{DefineW}
W(s) : =\Big( \int_{\S^2} w^2(s,\theta) d\theta \Big)^{1/2}
\end{equation}
where $y=(s,\theta), s=|y|=r^{-1}$ and  $w(s,\theta):=v(s,\theta)-\overline v(s)$,  $\overline v$ is the spherical average of $v$ on $\S^2$, i.e 
\begin{equation}\label{EqAverageOfV}
\overline v(s) = \frac1{|\S^2|}\int_{\S^2} v(s,\theta) d\theta.
\end{equation}
Hence, by exploiting further estimates for $v$ and $\overline v$ near $s=0$ in the next section, we  deduce the asymptotic expansion at infinity of $u$ and $\Delta u$ in Theorem \ref{ThmExpansion}, which is crucial in our argument. In the final step, we transform \eqref{MainEq} into a system of  two partial differential equations as follows 
\begin{equation}\label{systemPDE}
\begin{cases}
-\Delta u &=w \quad \text{in } \R^3,\\
-\Delta w &=-u^{-q} \quad \text{in } \R^3
\end{cases}
\end{equation}
and apply the method of moving plane to the system \eqref{systemPDE} to ensure the radial property. Eventually, we conclude the main theorem by  noticing that the sufficiency follows from \cite[Theorem 1.3]{Gue12}.

\section{Preliminaries}\label{SectPreparation}

In this section, we first state here the basic properties of  the Laplace operator on $\S^2$. It is well--known from \cite{CH62} that the eigenvalues of the operator $-\Delta_{\S^2}$ are given by
\[
\lambda_k = k(k+1) \quad (k \in \N),
\]
with the multiplicity  $m_k= 2k+1$ and we will denote the corresponding eigenfunctions  by $Q_1^j,Q_2^j,\dots,Q_{m_k}^k$. Without restricting the generality, we assume that 
\[
\{Q_1^0(\theta),Q_1^1(\theta),\dots,Q_{m_1}^1(\theta),Q_1^2(\theta),\dots,Q_{m_2}^2(\theta),Q_1^3(\theta),\dots \}
\]
 is a standard normalized basis of $H^2(\S^2)$. As indicated in \cite[Lemma 2.1]{GHZ15}, the eigenvalues of $\Delta_{\S^2}^2$ are of the form $\lambda_k^2 (k\in \N)$ with the same multiplicity. Hence, we obtain 
 \[
 \int_{\S^2} |\nabla_{\theta} w|^2 d\theta \geqslant 2 \int_{\S^2} w^2 d\theta
 \]
 and
 \[
 \int_{\S^2} |\Delta_{\theta} w|^2 d\theta \geqslant 4 \int_{\S^2} w^2 d\theta
 \]
 for any function $w$ orthogonal to $Q_1^0$. Thanks to the bootstrap argument, we deduce that 
\[
\max_{\theta \in \S^2}|Q_j^k(\theta)|\leqslant D_k, \quad \max_{\theta \in \S^2} |\nabla_{\theta} Q_j^k(\theta)| \leqslant E_k
\]
for $1\leqslant j \leqslant m_k$, where 
\begin{equation}\label{DefDkEk}
D_k :=C(1+\lambda_k+\lambda_k^2+\dots+\lambda_k^{\tau_1}), \quad E_k:=C(1+\lambda_k+\lambda_k^2+\dots+\lambda_k^{\tau_2})
\end{equation}
with a positive constant $C$ independent of $k$ and $\tau_1,\tau_2 \in \N$ greater than 2.   

Recall here the Kelvin transform \eqref{KelvinTransform}, we obtain that the function $v$  satisfies
\begin{equation}\label{MainEqTransformODEv}
\partial_s^4 v + 4s^{-1} \partial_s^3 v+  2 s^{-4}\Delta_{\theta} v + 2s^{-2} \Delta_{\theta}(\partial_s^2 v)+s^{-4} \Delta_{\theta}^2 v+s^{q-7}(v+\kappa)^{-q}=0,
\end{equation}
which is a consequence of the following computation
\[
\Delta_x^2 u = \Big( \partial_r^4 +4 r^{-1} \partial_r^3 + 2r^{-4} \Delta_{\theta}+ 2r^{-2}\Delta_{\theta} \partial_r^2+ r^{-4}\Delta_{\theta}^2 \Big) u.
\]
Furthermore, there exists  two positive constants $M$ and $s^{\star}$ depending only on $u$ such that 
\begin{equation}\label{LimitVy}
\lim_{|y|\to 0} v(y)=0, \quad |\nabla^l v(y)| \leqslant \frac{M}{s^l} \quad \text{for } s=|y|\leqslant s^{\star},
\end{equation}
due to the standard elliptic theory.  Next, a direct calculation shows that $\overline v$ and $w$ respectively fulfill 
\begin{equation}\label{ODEaverageV}
\partial_s^4 \overline v(s) +4s^{-1} \partial_s^3 \overline v +s^{q-7} \overline{(v+L)^{-q}}=0,
\end{equation}
and 
\begin{equation}\label{PDEw}
\partial_s^4 w + 4s^{-1}\partial_s^3 w+ 2s^{-4} \Delta_{\theta} w + 2s^{-2} \Delta_{\theta}(\partial_s^2 w)+ s^{-4}\Delta_{\theta}^2 w - s^{-4} g(w)=0,
\end{equation}
where 
\[
\begin{split}
g(w)  &= s^{q-3} (v+L)^{-q}- s^{q-3} \overline{(v+L)^{-q}}\\
&= -qs^{q-3} \Big[ (\xi(s,\theta)+L)^{-q-1}w(s,\theta) - \overline{(\xi(s,\theta)+L)^{-q-1}w(s,\theta)} \Big]
\end{split}
\]
and  $\xi(s,\theta)$ is between $v(s,\theta)$ and $\overline v(s)$. Let denote
\[
\begin{split}
\zeta(s) &= \max_{\theta \in \S^2} |g(w)| \\
&= \max_{\theta \in \S^2} \Big| -qs^{q-3} \Big[ (\xi(s,\theta)+L)^{-q-1}w(s,\theta) - \overline{(\xi(s,\theta)+L)^{-q-1}w(s,\theta)} \Big] \Big|.
\end{split}
\]
We see that $\zeta(s)=O(s^{q-3})$ and $\xi(s,\theta) \to 0$ as $s \to 0$.

\section{An upper bound of $W(s)$ for $s$ small}\label{SectProperties}

This section is devoted to give a priori estimate of $W(s)$, introduced in \eqref{DefineW}. We prove the following proposition.
\begin{proposition}\label{PropEstimateW}
There exists a sufficiently small $s_0$ and $C >0$ independent of $s_0$ such that for $s \in(0,s_0)$ 
\begin{equation}\label{EstimateW}
W(s) \leqslant Cs.
\end{equation}
\end{proposition}
\begin{proof}
It is worth noting that $w \in H^2(\S^2) \subset L^2(\S^2)$ and $\overline w(s)= 0$. Then, we have the expansion 
\begin{equation}\label{ExpansionOfW}
w(s,\theta)= \sum_{k=1}^{\infty} \sum_{j=1}^{m_k} w_j^k(s) Q_j^k(\theta).
\end{equation}
Substituting \eqref{ExpansionOfW} into \eqref{PDEw}, we deduce that  $w_j^k(s)$ with $1 \leqslant j \leqslant m_k$ is such that 
\begin{equation}\label{ODEwjk}
\partial_s^4 w_j^k +4s^{-1}  \partial_s^3 w_j^k -2\lambda_ik s^{-2} \partial_s^2 w_j^k - (2\lambda_k -\lambda_k^2) s^{-4} w_j^k = s^{-4} g_j^k(s),
\end{equation}
where 
\[
g_j^k (s)= \int_{\S^2} g(w) Q_j^i(\theta) d\theta = \int_{\S^2} f'(\xi(s,\theta)) w(s,\theta) Q_j^k(\theta) d\theta.
\]
which is bounded by 
\[
|g_j^k(s)| \leqslant C \zeta(s)W(s) = O(s^{q-3})W(s) 
\]
for $s$ near 0. Furthermore, to prove \eqref{EstimateW}, we only need to consider the case 
\[
|g_j^k(s)|=o_s(1)|w_j^k(s)|.
\]
Indeed,  note that $g_j^k(s)$ and $w_j^k(s)$ are Fourier coefficients of $f'(\xi)w(s,\theta)$ and $w(s,\theta)$, respectively.   One obtains
\[
\| f'(\xi) w(s,\theta) \|_{L^2(\S^2)} \leqslant \zeta(s) \|w(s,\theta)\|_{L^2(\S^2)} = o_s(1) \|w\|_{L^2(\S^2)},
\]
which yields 
\[
\sum_{k=1}^{\infty} \sum_{j=1}^{m_k} |g_j^k(s)|^2 = o_s(1)  \sum_{k=1}^{\infty} \sum_{j=1}^{m_k} |w_j^k(s)|^2.
\]
Set 
\[
\begin{split}
G_s &= \{ (j,k) : k\geqslant 1, 1\leqslant j\leqslant m_k \text{ such that } |g_j^k(s)|=o_s(1) |w_j^k(s)| \},\\
B_s &= \{ (j,k) : k\geqslant 1, 1\leqslant j\leqslant m_k \text{ such that } |g_j^k(s)| \neq o_s(1) |w_j^k(s)| \}.
\end{split}
\]
Now, we show that there exists  $0<\tilde s<s^{\star}$ ($s^{\star}$ is given in  \eqref{LimitVy}) and $C>0$ independent of $j,k$ and $s$ such that for any $0<s<\tilde s$  we have 
\[
|g_j^k(s)| \geqslant C |w_j^k(s)|
\]
for $(j,k) \in B_s$. Indeed, we contradicts that there exists $c_n \to 0$ and $s_n\to 0$ as $n\to \infty$ such that 
\[
|g_{j_n}^{k_n}(s_n)| \leqslant c_n |w_{j_n}^{k_n}(s_n)|
\]
for large $n$ and $(j_n,k_n) \in B_{s_n}$. Then, 
\[
|g_{j_n}^{k_n}(s_n)| \leqslant o_{s_n}(1) |w_{j_n}^{k_n}(s_n)|
\]
for large $n$, which contradicts $(j_n,k_n) \in B_{s_n}$. Hence, for any $s\in(0,\tilde s)$,
\[
\sum_{(j,k)\in B_n} |w_j^k(s)|^2 \leqslant C^{-2} \sum_{(j,k) \in B_n} |g_j^k(s)|^2 = o_s(1) \sum_{k=1}^{\infty} \sum_{j=1}^{m_k} |w_j^k(s)|^2.
\]
Then, we  assume that  in the rest of the proof of Proposition \ref{PropEstimateW}, there holds
\begin{equation}\label{RelationGjkWjk}
|g_j^k(s)|=o_s(1)|w_j^k(s)|
\end{equation}
for any $k\geqslant 1, 1\leqslant j \leqslant m_k$ and $0<s<\tilde s$.

Let $t=-\ln s$ and $ z_j^k(t) = w_j^k(s)$. Then, \eqref{ODEwjk} becomes
\begin{equation}\label{ODEzjk}
\partial_t^4 z_j^k +2\partial_t^3 z_j^k - (1+2\lambda_k) \partial_t^2 z_j^k - 2(1+\lambda_k) \partial_t z_j^k + \lambda_k(\lambda_k-2) z_j^k = f_j^k(t),
\end{equation}
where $f_j^k(t) = g_j^k(e^{-t})$. The corresponding characteristic polynomial of \eqref{ODEzjk} is 
\[
\mu^4+2\mu^3-(1+2\lambda_k)\mu^2-2(1+\lambda_k)\mu +\lambda_k(\lambda_k-2)=0,
\]
which has the following roots
\[
\mu_1^{(k)}=-k-2, \quad \mu_2^{(k)}=-k, \quad \mu_3^{(k)}=k-1, \quad \mu_4^{(k)} = k+1.
\]
Due to the variation of parameters formula, we will show that 
\begin{equation}\label{BigOz}
|z_j^k(t)|=O(e^{-kt})
\end{equation}
by considering two following cases.

\noindent
\textbf{Case 1: $k \geqslant 2$}.
Due to the fact that $z_j^k(t)$ tends to 0 as $t$ tends to $\infty$, one obtains
\[
\begin{split}
z_j^k(t) &= A_{j,1}^k e^{-(k+2)t}+ A_{j,2}^i e^{-k t}\\
&+B_4^k\int_t^{\infty}e^{(k+1)(t-\tau)}f_j^k(\tau) d\tau + B_3^k\int_t^{\infty}e^{(k-1)(t-\tau)}f_j^k(\tau) d\tau\\
&+B_1^k \int_T^t e^{-(k+2) (t-\tau)} f_j^k(\tau)d\tau + B_2^k \int_T^t e^{-k (t-\tau)} f_j^k(\tau)d\tau.
\end{split}
\]
Hence, there exists a constant $C$ depending only on $B_i^k (i=1,2,3,4)$ such that
\begin{equation}\label{EstimateZjk1}
|z_j^k(t)|\leqslant O(e^{-k t})+C \int_t^{\infty}e^{(i-1) (t-\tau)}|f_j^k(\tau)| d\tau +C\int_T^t e^{-k (t-\tau)} |f_j^k(\tau)|d\tau.
\end{equation}
In view of \eqref{RelationGjkWjk}, we obtain that  $|f_j^k(t)|=o_t(1) |z_j^k(t)|$ for $t\in (T,\infty)$. Then, we substitute it into \eqref{EstimateZjk1} to get that 
\begin{equation}\label{EstimateZjk2}
\begin{split}
|z_j^k(t)| \leqslant O(e^{-k t}) &+C \int_t^{\infty}e^{(k-1) (t-\tau)}o_{\tau}(1) |z_j^k(\tau)| d\tau\\
 &+ C\int_T^t e^{-k (t-\tau)} o_{\tau}(1) |z_j^k(\tau)|d\tau.
\end{split}
\end{equation}
Note that for  any $\varepsilon$ small enough, there exists $t$ large such that $o_t(1) <\epsilon$. Let 
\[
K_1(t) = \int_t^{\infty}e^{(k-1) (t-\tau)} |z_j^k(\tau)| d\tau,
\]
and 
\[
K_2(t) = \int_T^t e^{-k (t-\tau)} |z_j^k(\tau)| d\tau.
\]
Clearly $\lim\limits_{t \to \infty} K_1(t)= \lim\limits_{t \to \infty} K_2(t)=0$ by $\lim\limits_{t\to \infty }z_j^k(t) = 0$.
Hence, 
\[
\begin{split}
(K_2-K_1)'(t)&=2|z_j^k(t)| -(k-1)K_1(t) - k K_2(t)\\
&\leqslant 2C\varepsilon (K_1(t)+K_2(t))-(k-1)K_1(t) - k K_2(t)+O(e^{-k t}) \\
&\leqslant O(e^{-k t}),
\end{split}
\]
which yields 
\[
K_1(t)-K_2(t) \leqslant O(e^{-kt}),
\]
or 
\[
K_1(t) \leqslant K_2(t) +O(e^{-kt}).
\]
This implies that
\[
K_2'(t) = |z_j^k(t)| - k K_2(t) \leqslant O(e^{-kt}) + (2C\varepsilon -k)K_2(t).
\]
Thus,  $K_2(t)= O(e^{(2C\varepsilon -k)t})$. Plugging it into \eqref{EstimateZjk2}, one gets that $|z_j^k(t)|= O(e^{-kt})$. 

\noindent
\textbf{Case 2: $k =1$}.
One has
\[
\mu_1^{(1)} = -3< \mu_2^{(1)} =-1 <\mu_3^{(1)}=0<\mu_4^{(1)}<2.
\]
We also obtain that
\[
\begin{split}
z_j^1(t)&= A_{j,1}^1 e^{-t}+A_{j,3}^1 e^{-3t}-B_4^1 \int_t^{\infty}e^{2(t-\tau)}f_j^1(\tau)d\tau\\
& - B_3^1 \int_t^{\infty} f_j^1(\tau)d\tau + B_2^1 \int_T^t e^{-(t-\tau)} f_j^1(\tau)d\tau +B_4^1 \int_T^t e^{-3(t-\tau)} f_j^1(\tau)d\tau.
\end{split}
\]
It then follows that
\begin{equation}\label{EqODEestimateZj1}
|z_j^1(t)|=O(e^{-(t-T)})+C \int_t^{\infty}|o_s(1) z_j^1(\tau)|d\tau,
\end{equation}
for $1\leqslant j \leqslant m_1$ and $t>T$.
Under the assumption \eqref{MainCondition}, one gets $|w(s,\theta)|^2 \leqslant C s^{2\vartheta}$ for $s$ near 0. Thus, 
\[
\int_0^s \xi^{-1} |w_j^1(\xi)|  d\xi  \leqslant \int_0^s \xi^{-1} \Big( \int_{\S^2} w^2(s,\theta) d\theta \Big)^{1/2} d\xi  \leqslant C \int_0^s \xi^{-(1-\vartheta)} d\xi < \infty.
\]
Equivalently, we have just shown that $\int_t^{\infty} |z_j^1(\tau)| d\tau < \infty$. 
Let us define 
\[
K(t) = \int_t^{\infty} |z_j^1(\tau)| d\tau,
\]
that gives
\[
-K'(t) = |z_j^1(t)| \leqslant O(e^{-t})+ C\varepsilon K(t),
\]
which yields $K(t)=O(e^{-t})$. From this, we turn back to \eqref{EqODEestimateZj1} to get that  $|z_j^1(t)|=O(e^{-t})$.

Thanks to \eqref{BigOz}, one has
\[
 \sum_{k=1}^{\infty} \sum_{j=1}^{m_k} |z_j^k(t)| \leqslant O(e^{-t})+ O\Big(\sum_{k=2}^{\infty} km_k e^{-kt}\Big).
\] 
For $t>T_1 >T$, one has
\begin{equation}\label{LimitKmK}
\sum_{k=2}^{\infty} km_ke^{-k(t-T)} = O(e^{-2(t-T)}),
\end{equation}
by observing 
\[
\lim_{k\to \infty} \frac{(k+1)m_{k+1}e^{-(k+1)(t-T)}}{km_ke^{-k(t-T)}} = e^{-(t-T)} \lim_{k\to \infty} \frac{(k+1)m_{k+1}}{km_k} = e^{-(t-T)} < \frac12.
\]
Then, 
\[
 \sum_{k=1}^{\infty} \sum_{j=1}^{m_k} |z_j^k(t)| \leqslant O(e^{-t}),
 \]
 which implies \eqref{EstimateW} for $0<s<s_0 = e^{-T_1}$. Proof of Proposition \ref{PropEstimateW} is complete.
\end{proof}

\section{Expansion of $u$ at infinity}\label{SectEstimateV}

In this section, the main result is given in Theorem \ref{ThmExpansion} below. To carry out our analysis,  the following propositions are needed.
\begin{proposition}\label{PropEstimateVbar}
Let $v$ be solution of Eq. \eqref{MainEqTransformODEv}. Then, there exist positive constants $M=M(v)$ and $\varpi \in (0,1/10)$ such that 
\begin{equation}\label{EstimateVbar}
\begin{cases}
|\overline v(s)|\leqslant Ms^q, \quad |\overline v'(s)|\leqslant Ms^{q-1}, \quad |\overline v''(s)|\leqslant Ms^{q-2} 
&\quad\text{if } 3<q<4,\\
|\overline v(s)|\leqslant Ms^{1-\varpi}, \quad |\overline v'(s)|\leqslant M s^{-\varpi}, \quad |\overline v''(s)|\leqslant Ms^{-1-\varpi} &\quad\text{if } q\geqslant4
\end{cases}
\end{equation}
for $s$ small enough. Furthermore, there also holds 
\begin{equation}\label{EstimateIntegralVsquare}
\int_{\S^2} v^2(s,\theta) d\theta \leqslant
\begin{cases}
Ms^{2(q-3)} &\quad\text{if } 3<q<4,\\
Ms^{2(1-\varpi)} &\quad\text{if } q\geqslant 4.
\end{cases}
\end{equation}

\end{proposition}
\begin{proof}
Similar to Proposition \ref{PropEstimateW}, we also transform \eqref{ODEaverageV}  into
\begin{equation}\label{ODEzBar}
\partial_t^4 \overline z+2\partial_t^3\overline z-\partial_t^2 \overline z-2\partial_t \overline z= h(\overline z)+O(e^{-t}),
\end{equation}
where  $t=-\ln s, \overline z(t)=\overline v(s)$ and
\begin{equation}\label{BigOhZ}
h(\overline z)= s^{q-3}(\overline z+L)^{-q}=O(e^{-(q-3)t}).
\end{equation}
The corresponding characteristic polynomial of \eqref{ODEzBar} is 
\[
\mu^4+2\mu^3-\mu^2-2\mu=0,
\]
which has four roots $-1,0,1$ and $-2$. Notice that $\overline z(t)$ tends to 0 as $t$ tends to infinity, hence 
\begin{equation}\label{ODEtransformOverZ}
\begin{split}
\overline z(t) &= A_1 e^{-t}+A_2 e^{-2t}\\
&+B_1 \int_t^{\infty} e^{(t-\tau)}\overline{f}(\tau) d\tau +B_2  \int_t^{\infty} \overline{f}(\tau) d\tau\\
&+ B_3 \int_T^{t} e^{-(t-\tau)}\overline{f}(\tau) d\tau+ B_4 \int_T^{t} e^{-2(t-\tau)}\overline{f}(\tau) d\tau,
\end{split}
\end{equation}
where $\overline{f}(t)=h(\overline{z}(t))+ O(e^{-t})$.  In addition, \eqref{BigOhZ} leads us to 
\begin{equation}\label{BigOoverF}
\overline f(t) = O(e^{-\min\{q-3,1\}t}).
\end{equation}
 Plugging \eqref{BigOoverF} into \eqref{ODEtransformOverZ}, one has
\begin{equation}
|\overline z(t)| \leqslant
\begin{cases}
Me^{-(1-\varpi)t} &\quad\text{if } q\geqslant 4,\\
Me^{-(q-3)t} &\quad\text{if } 3<q<4
\end{cases}
\end{equation}
for large $t$ and sufficiently small $\varpi$, which is equivalent to
\begin{equation}\label{BoundOverlineV}
 |\overline v(s)| \leqslant
\begin{cases}
Ms^{1-\varpi} &\quad\text{if } q\geqslant 4,\\
Ms^{q-3} &\quad\text{if } 3<q<4
\end{cases}
\end{equation}
for small $s$. In order to deduce the rest of \eqref{EstimateVbar}, we caculate the first and second derivative $\overline z(t)$ and notice  that $\overline v'(s)=-\overline z'(t) e^t$ and that $\overline v''(s) = (\overline z''(t)+\overline z'(t)) e^{2t}$. \eqref{EstimateIntegralVsquare} is obtained by Proposition \ref{PropEstimateW} and 
\eqref{BoundOverlineV}.
\end{proof}
\begin{proposition}\label{PropEstimateDtauV}
Let $\tau \geqslant 0$ be an integer, $v$ be a solution of \eqref{MainEqTransformODEv} and $\varpi$ be defined as in Proposition \ref{PropEstimateVbar}. Then, there exists  $M=M(v,\tau)>0$  such that for $s$ small enough,
\begin{equation}\label{EstimateDtauV}
\max_{|y|=s} |D^{\tau} v(y)| \leqslant 
\begin{cases}
Ms^{1-\varpi-\tau} &\quad\text{if } 3<q<4,\\
Ms^{1-\varpi-\tau} &\quad\text{if } q\geqslant 4.
\end{cases}
\end{equation}
\end{proposition} 
\begin{proof}
For the case $\tau =0$, we obtain
\[
\max_{\theta \in \S^2} |z(t,\theta)| \leqslant \sum_{k=1}^{\infty} \sum_{j=1}^{m_k} |z_j^k(t)| \max_{\theta\in \S^2} |Q_j^k(\theta)| \leqslant \sum_{k=1}^{\infty} \sum_{j=1}^{m_k} D_k |z_j^k(t)|,
\]
where $z(t,\theta)=w(s,\theta)$ and $D_k$ is given in \eqref{DefDkEk}. Note that 
\[
\lim_{k\to \infty} \frac{(k+1)m_{k+1}D_{k+1}}{km_kD_k} =1,
\]
then a same argument of that in the proof of Proposition \ref{PropEstimateW} implies that there exists positive $C$ independent of $t$ and large $\tilde T$ such that for $t>\tilde T$, 
\[
\sum_{k=1}^{\infty} \sum_{j=1}^{m_k} D_k |z_j^k(t)| = O \Big( \sum_{k=2} km_kD_ke^{-k(t-T)} \Big)+O(e^{-(t-T)}) \leqslant C e^{-t},
\]
which yields 
\[
\max_{\theta \in \S^2} |z(t,\theta)| \leqslant C e^{-t}.
\]
Equivalently, 
\[
\max_{\theta \in \S^2} |w(s,\theta)|\leqslant Cs
\]
for $0<s<e^{-\tilde T}$.  Combining with \eqref{EstimateVbar}, this guarantees \eqref{EstimateDtauV} when $\tau=0$. For the case $\tau =1$, we see that 
\[
|\nabla w|^2 = w_s^2+ \frac1{s^2} |\nabla_{\theta} w|^2.
\] 
One has 
\[
w_s(s,\theta)= \sum_{k=1}^{\infty} \sum_{j=1}^{m_k} (w_j^k)'(s) Q_j^k(\theta).
\]
Hence, 
\[
\max_{\theta \in \S^2} |w_s(s,\theta)| \leqslant \sum_{k=1}^{\infty} \sum_{j=1}^{m_k} D_k |(w_j^k)'(s)|.
\]
Notice that $(w_j^k)'(s)=-(z_j^k)'(t)e^t$, hence 
\[
\sum_{k=1}^{\infty} \sum_{j=1}^{m_k} D_k |(w_j^k)'(s)| \leqslant O(1) + O \Big( \sum_{k=2}^{\infty} km_kD_k s^{k-1}).
\]
This implies that there exists $M_1=M_1(v)$ such that  for $s$ sufficiently small,
\begin{equation}\label{MaxWs}
\max_{\theta \in S^2} |w_s(s,\theta)|\leqslant M_1.
\end{equation}
We also obtain that
\[
|\nabla_{\theta} w(s,\theta)| \leqslant \sum_{k=1}^{\infty} \sum_{j=1}^{m_k} |w_j^k(s)| |\nabla_{\theta} Q_j^k| \leqslant \sum_{k=1}^{\infty} \sum_{j=1}^{m_k} E_k  |w_j^k(s)|,
\]
and notice that
\[
\lim_{k\to \infty} \frac{(k+1)m_{k+1}E_{k+1}}{km_kE_k} =1.
\]
Consequently, one gets that there also exists $M_2=M_2(v)$ such that for $s$ sufficiently small,
\begin{equation}\label{MaxWtheta}
\max_{\theta \in \S^2} |\nabla_{\theta} w(s,\theta) | \leqslant M_2 s.
\end{equation}
Thanks to \eqref{MaxWs} and \eqref{MaxWtheta}, we deduce that 
\[
|\nabla_{\theta} w| \leqslant M_1^2+M_2^2,
\]
i.e. \eqref{EstimateDtauV} holds for $\tau=1$. By differentiating $w(s,\theta)$, we conclude the other cases of $\tau$. Proof of Proposition \ref{PropEstimateDtauV} is complete.
\end{proof}
\begin{lemma}\label{LemAsymptoticW1}
Let $\tilde w(s,\theta)= w(s,\theta)/s$, there holds
\begin{equation}
\lim_{s\to 0} \tilde w(s,\theta)= V(\theta),
\end{equation}
where $V(\theta)= \theta \cdot x^{\star}$ for some $x^{\star} \in \R^3$ fixed and $\theta \in \S^2$.
\end{lemma}
\begin{proof}
By computing directly, it follows from \eqref{PDEw} that $\tilde w$ is the solution of the following equation 
\begin{equation}\label{OdeTransformTildeW}
\begin{split}
&\partial_s^4 \tilde{w}+8s^{-1}\partial_s^3 \tilde{w} + 12s^{-2} \partial_s^2 \tilde{w} +2s^{-4}\Delta_{\theta} \tilde{w}\\
&\quad \quad+ 4s^{-2}\Delta_{\theta}(\partial_s \tilde{w})+ 2s^{-3} \Delta_{\theta}(\partial_s^2 \tilde w)+ s^{-4} \Delta_{\theta}^2 \tilde w = s^{-4} g(\tilde w),
\end{split}
\end{equation}
where 
\[
g(\tilde w) = -qs^{q-3} \Big[ (\xi(s,\theta)+L)^{-q-1} \tilde w(s,\theta) - \overline{(\xi(s,\theta)+L)^{-q-1} \tilde w(s,\theta)} \Big]
\]
and $\xi(s,\theta)$ is between $v(s,\theta)$ and $\overline v(s)$. Notice that $\overline{\tilde w} = \overline w /s =0$, hence
\[
\tilde w(s)= \sum_{k=1}^{\infty} \sum_{j=1}^k \tilde{w}_j^k(s) Q_j^k(\theta),
\]
where $\tilde{w}_j^k(s) = w_j^k(s)/s$. A direct computation shows that 
\begin{equation}\label{OdeTildeWjk}
\partial_s^4 \tilde{w}_j^k +8s^{-1} \partial_s^3 \tilde{w}_j^k +(12-2\lambda_k) s^{-2} \partial_s^2 \tilde{w}_j^k - 4 \lambda_k s^{-3} \partial_s \tilde{w}_j^k+ (\lambda_k^2-2\lambda_k) s^{-4}\tilde{w}_j^k = s^{-4} \tilde g_j^i(s),
\end{equation}
where $\tilde g_j^k(s)= \int_{\S^2}g(\tilde w)Q_j^k(\theta)d\theta$.  Notice that 
\begin{equation}\label{EstimateTildeG}
|\tilde{g}_j^k(s)| \leqslant O(s^{q-3})\tilde W(s),
\end{equation}
where $\tilde W(s) = \Big( \int_{\S^2} |\tilde w(s,\theta)|^2 d\theta \Big)^{1/2}$.

Next, we repeat the process used in Proposition \ref{PropEstimateW}. Making change of variable $t=-\ln s$ and letting $ \tilde z_j^k(t) = \tilde w_j^k(s)$. Hence
\begin{equation}\label{OdeTildeZjk}
\partial_t^4 \tilde z_j^k - 2\partial_t^3 \tilde z_j^k - (1+2\lambda_k) \partial_t^2 \tilde z_j^k + 2(1+\lambda_k) \partial_t\tilde  z_j^k + \lambda_k(\lambda_k-2) \tilde z_j^k = \hat g_j^k(t),
\end{equation}
where $\hat g_j^k(t) = \tilde g_j^k(s)$.   The corresponding characteristic polynomial of \eqref{OdeTildeZjk} is 
\[
\tilde{\mu}^4-2\tilde{\mu}^3-(1+2\lambda_k)\tilde{\mu}^2+2(1+\lambda_k)\tilde{\mu} +\lambda_k(\lambda_k-2)=0,
\]
which has 4 real roots as follows
\[
\tilde{\mu}_1^{(k)}=-k-1, \quad \tilde{\mu}_2^{(k)}=-k+1, \quad \tilde{\mu}_3^{(k)}=k, \quad \tilde{\mu}_4^{(k)} = k+2.
\]
It follows from \eqref{BigOz} that $\lim\limits_{s \to 0} \tilde{w}_j^k(s) =0$
for $k \geqslant 2$  and that  $\tilde w_j^1(s)$ is bounded for $s$ near 0.
Hence, $\tilde z_j^1(t)$ is bounded for $t$ near infinity. We also note that $|\tilde W(s)|=O(1)$ for $s$ near 0 from \eqref{EstimateW}. Thus, \eqref{EstimateTildeG} tells us that $|\tilde g_j^k(s)|=O(s^{q-3})$, which means $|\hat g_j^k(t)|=O(e^{-(q-3)t})$ for $t$ near infinity. Combining those  above facts, one gets for  $t>T$ large enogh
\[
\begin{split}
z_j^1(t) &= C+ A e^{-2t}+ B_1 \int_t^{\infty} e^{3\tau} O(e^{-(q-3)\tau}) d\tau + B_2 \int_t^{\infty} e^{\tau} O(e^{-(q-3)\tau}) ds \\
&+ B_3  \int_t^{\infty} O(e^{-(q-3)\tau})d\tau + B_4 \int_T^t e^{-2\tau} O(e^{-(q-3)\tau})d\tau.
\end{split}
\]
This implies that $\tilde z_j^1(t)$ converges to a constant as $t \to \infty$, i.e. $\tilde w_j^1(s)$ converges to a constant as $s\to 0$ for all $1\leqslant j\leqslant m_1$. Note that $Q_1^1(\theta), Q_2^1(\theta)$, \dots and $ Q_{m_1}^1(\theta)$ are the eigenfunctions corresponding to the eigenvalue $\lambda_1=2$. Thus, we deduce that 
\[
\lim_{s\to 0} \tilde w(s,\theta) = V(\theta),
\]
where $V(\theta) \equiv 0$ or one of the first eigenfunctions of $-\Delta_{\S^2}$. From the well--known result \cite[Lemma 8.1]{Zou95}, we conclude that $V(\theta)=\theta \cdot x^{\star}$ for some $x^{\star} \in \R^3$ fixed. 
\end{proof}
We sum up the previous results.
\begin{theorem}\label{TheoSumUp}
Let $v$ be a solution of \eqref{MainEqTransformODEv} and $\tilde w$ defined in Lemma \ref{LemAsymptoticW1}.  Then, $v(y) = \overline v(s)+ s\tilde w(s)$, with $\overline v$ and $\tilde w$ satisfy the following properties.
\begin{enumerate}
\item For $s$ sufficiently small,
\[
|\overline v(s)|= O(s^{1-\varpi}), \quad |\overline v'(s)|=O(s^{-\varpi}),  \quad |\overline v''(s)| =O(s^{-1-\varpi})  \quad q\geqslant 4,
\]
or
\[
|\overline v(s)| =O(s^{q-3}), \quad |\overline v'(s)|  =O(s^{q-4}),  \quad |\overline v''(s)|=O(s^{q-5}) \quad 3<q<4.
\]
\item For any nonnegative integers $\tau_1$ and $\tau_2$,
\[
|D_{\theta}^{\tau_1} D_s^{\tau_2} \tilde w(s,\theta)| = O(s^{-\tau_2}).
\]
Furthermore, $\tilde w(s,\theta)$ converges uniformly in $C^{\tau_2}(\S^2)$  to $V(\theta)$ which is 0 or one of the first eigenfunctions of $\Delta_{\theta}$ on $\S^2$ as $s\to 0$.
\end{enumerate}
\end{theorem}
With Theorem \ref{TheoSumUp} in hand, we are able to obtain the asymptotic expansion of the solution $u$ of \eqref{MainEq}.
\begin{theorem}\label{ThmExpansion}
Let $u$ be a solution of \eqref{MainEq} under the assumption \eqref{MainCondition}. Then, $u$ and $-\Delta u$ admit the expansion 
\begin{equation}\label{ExpansionOfUx}
\begin{cases}
u(x) &=r\Big( L+\xi(r)+\frac{\eta(r,\theta)}r\Big),\\
-\Delta u(x)&= -\frac1r \Big(2L+\xi_1(r) +\frac{\eta_1(r,\theta)}r \Big).
\end{cases}
\end{equation}
at infinity, where 
\[
\begin{cases}
\xi_1(r) &= -\Big( r^2 \xi''+6r \xi'+6\xi \Big),\\
\eta_1(r,\theta) &= -\Big( r^2 \partial_r^2 \eta + 4r \partial_r \eta+ 2\eta +\Delta_{\theta} \eta \Big).
\end{cases}
\]
The functions $\xi$ and $\eta$ satisfy the following properties. 
\begin{enumerate}
\item $\xi(r) = r^{-1}\overline u(r) -L$ and for $r$ large enough, we have
\[
|\xi(r)| =O(r^{-1+\varpi}), \quad |\xi'(r)| = O(r^{-2+\varpi}), \quad |\xi''(r)|= O(r^{-3+\varpi}) \quad \text{if } q\geqslant 4,
\]
or
\[
|\xi(r)| =O(r^{-(q-3)}), \quad |\xi'(r)| = O(r^{-(q-2)}), \quad |\xi''(r)|= O(r^{-(q-1)}) \quad \text{if } 3<q<4.
\]
\item Let $\tau_1$ and $\tau_2$ be two non--negative integers. There exists a positive constant $M=M(u,\tau_1,\tau_2)$ such that 
\[
|r^{\tau_2} D_{\theta}^{\tau_1}D_r^{\tau_2} \eta(r,\theta)|\leqslant M \quad\text{and}\quad |\eta_1(r,\theta)| \leqslant M \quad\text{for } r \text{ large enough}.
\]
\item Let $\tau$ be a non-negative integer. Then $\eta(r,\theta)$ tends to $V(\theta)$ uniformly in $C^{\tau}(\S^2)$ as $r \to \infty$, where $V(\theta)$ is given in Lemma \ref{LemAsymptoticW1}.
\end{enumerate}
\end{theorem}

\section{Proof of Theorem \ref{MainThm}}\label{SectMainProof}
In order to give the proof of Theorem \ref{MainThm}, we use the method of moving plane. For $\gamma \in \R$, let us define the hyperplane $\Gamma_{\gamma} = \{ x=(x_1,x_2,x_3) \in \R^3, x_1=\gamma \}$. For any $x\in \R^3$, we denote  the reflection point of $x \in \R^3$ about $\Gamma_{\gamma}$ by $x^{\gamma}$ and the first component of $x \in \R^3$ by $(x)_1$. Our next lemma is a consequence of Theorem \ref{ThmExpansion} and the proof follows from \cite[Lemma 8.2]{Zou95} with a slight modification.
\begin{lemma}\label{LemLimitExpansionU}
Let $u$ be solution of \eqref{MainEq} under the assumption \eqref{MainCondition}.
\begin{enumerate}
\item  Let $\{\gamma_j\}$ be a real sequence that converges to $\gamma \in \R \cup \{\infty\}$ and $\{x^j\}$ be an unbounded sequence in $\R^3$ with $(x^j)_1 <\gamma_j$ for all $j\geqslant 1$. There holds,
\begin{equation}\label{LimitExpansionU}
\lim_{j\to\infty} \frac{|x^j|}{\gamma_j -(x^j)_1} \Big( u(x^j)-u((x^j)^{\gamma_j}) \Big) = -2L\gamma-2(x_0)_1,
\end{equation}
where  $x_0$ is given in Lemma \ref{LemAsymptoticW1}.
\item We have
\[
\frac{\partial u}{\partial x_1} \geqslant 0 \quad \text{if } x_1 \geqslant \gamma_0+1 \text{ and } |x|\geqslant M,
\]
for some constants $M=M(u)$, where 
\[ 
\gamma_0 = - \frac{(x_0)_1}{2L}.
\]
\end{enumerate} 
\end{lemma}
\begin{proof}
\noindent
Part \textbf{(1)}.
Without restricting the generality, we assume that
\[
\lim_{j \to \infty} \frac{x^j}{|x^j|} =\overline \theta \in \S^2,
\]
and that $\gamma_j =\gamma$ for all $j$ since the next arguments work equally well for the sequence $\{\gamma_j\}$. It follows from the expansion of $u$ in \eqref{ExpansionOfUx} that 
\begin{equation}\label{AsympExpansionU}
\begin{split}
\frac{|x^j|}{\gamma-(x^j)_1} \Big( u(x^j)-u((x^j)^{\gamma}) \Big) &= \underbrace{\frac{L|x^j|}{\gamma-(x^j)_1} \Big( |x^j|-|(x^j)^{\gamma}|  \Big)}_{:=I}\\
&+\underbrace{\frac{|x^j|}{\gamma_j-(x^j)_1} \Big( |x^j| \xi(|x^j|) - |(x^j)^{\gamma}| \xi( |(x^j)^{\gamma}|) \Big)}_{:=II}\\
&+\underbrace{\frac{|x^j|}{\gamma-(x^j)_1} \Big( \eta(|x^j|,\theta^j) - \eta( |(x^j)^{\gamma}|,(\theta^j)^{\gamma}) \Big)}_{:=III}.
\end{split}
\end{equation}
Computing directly, 
\[
|x^j| (|x^j|-|(x^j)^{\gamma}|) = \frac{4 |x^j| \gamma ((x^j)_1 -\gamma)}{|x^j|+|(x^j)^{\gamma}|},
\]
which implies 
\begin{equation}\label{LimitFirstTerm}
\lim_{j\to \infty} I =-4L\gamma.
\end{equation}
Here we have used the fact that $|x^j|/|(x^j)^{\gamma}| \to 1$ as $j \to \infty$. About $II$, there is $\beta_j$ between $|x^j|$ and $|(x^j)^{\gamma}|$ such that 
\[
II=\frac{|x^j|}{\gamma-(x^j)_1} \Big( |x^j| \xi(|x^j|) - |(x^j)^{\gamma}| \xi( |(x^j)^{\gamma}|) \Big) =   \frac{-4\gamma |x^j|}{|x^j|+|(x^j)^{\gamma}|}\Big( \xi(\beta_j)+ \beta_j \xi'(\beta_j) \Big).
\]
For large $j$, it can be seen from Theorem \ref{ThmExpansion} that 
\[
|\beta_j||\xi'(\beta_j)|+ |\xi(\beta_j)|  \leqslant 
\begin{cases}
2M |\beta_j|^{-(1-\varpi)} \leqslant 2M|x^j|^{-(1-\varpi)}\quad&\text{if } q\geqslant 4,\\
2M |\beta_j|^{-(q-3)}  \leqslant 2M |x^j|^{-(q-3)}\quad&\text{if } 3<q<4.
\end{cases}
\]
Hence, 
\begin{equation}\label{LimitSecondTerm}
II\leqslant
\begin{cases}
O(|x^j|^{-(1-\varpi)}) \to 0 \quad&\text{ for } q\geqslant 4,\\
O(|x^j|^{-(q-3)}) \to 0 \quad&\text{ for } 3<q<4.
\end{cases}
\end{equation}
We now deal with the last term $III$ in \eqref{AsympExpansionU} by splitting $III$ into two terms $III_1$ and $III_2$ as follows,
\[
III_1 :=  \frac{|x^j|}{\gamma-(x^j)_1} \Big( \eta(|x^j|,\theta^j) - \eta( |x^j|,(\theta^j)^{\gamma}) \Big)
\]
and 
\[
III_2 := \frac{|x^j|}{\gamma-(x^j)_1} \Big( \eta(|x^j|,(\theta^j)^{\gamma}) - \eta( |(x^j)^{\gamma}|,(\theta^j)^{\gamma}) \Big).
\]
Using the mean value theorem, there exists $\beta_j$ between $|x^j|$ and $|(x^j)^{\gamma}|$ such that 
\[
III_2= \frac{|x^j|(|x^j|-|(x^j)^{\gamma}|)D_r \eta(\beta_j,(\theta^j)^{\gamma})}{\gamma-(x^j)_1},
\]
which implies
\begin{equation}\label{LimitThirdTerm2}
|III_2| = \Big| \frac{-4\gamma|x^j|D_r \eta(\beta_j,(\theta^j)^{\gamma})}{|x^j|+|(x^j)^{\gamma}|} \Big| \leqslant O(|x^j|^{-1}) \to 0 \text{ as } j\to \infty. 
\end{equation}
To estimate $III_1$, we suppose that 
\[
\lim_{j\to\infty} \frac{(x^j)^{\gamma}}{|(x^j)^{\gamma}|} = \widehat\theta
\]
and consider two cases $\overline \theta \neq \widehat \theta$ and $\overline  \theta = \widehat \theta$. If $\overline \theta \neq \widehat \theta$, we observe that 
\[
\frac{(\gamma-(x^j)_1,0,0)}{|x^j|} \to \frac12 (\widehat \theta -\overline \theta) \text{ as } j\to\infty.
\]
Hence 
\[
\frac{\gamma-(x_j)_1}{|x^j|} \to \frac12 (\widehat \theta -\overline \theta)_1 \text{ as } j\to\infty,
\]
which yields
\begin{equation}\label{LimitThirdTerm1Case1}
III_1 \to \frac{2 (V(\overline \theta)-V(\widehat \theta))}{(\widehat \theta -\overline \theta)_1}  = \frac{2(\overline \theta -\widehat \theta) \cdot x_0}{(\widehat \theta -\overline \theta)_1} = -2(x_0)_1.
\end{equation}
If $\overline \theta =\widehat \theta$, we obtain that there exists a point $\beta^j$ between $\theta^j$ and $(\theta^j)^{\gamma}$ on a geodesic on $\S^2$ such that
\[
\begin{split}
&\eta(|x^j|,\theta^j) - \eta( |x^j|,(\theta^j)^{\gamma} \\
&= \nabla_{\theta} \eta(|x^j|,\beta_j) ( \theta^j - (\theta^j)^{\gamma}) \\
&= \nabla_{\theta} \eta(|x^j|,\beta_j) \cdot \Big( \frac{x^j}{|x^j|} - \frac{(x^j)^{\gamma}}{|x^j|} \Big) + \nabla_{\theta} \eta(|x^j|,\beta^j) 
\cdot (x^j)^{\gamma} \Big( \frac1{|x^j|} -\frac1{|(x^j)^{\gamma}|} \Big)\\
&= \frac{-2(\gamma-(x^j)_1)}{|x^j|}  \Big( \partial_{\theta_1} \eta(|x^j|,\beta^j) + O \Big( \frac1{|x^j|} \Big)\Big).
\end{split}
\]
Consequently, 
\begin{equation}\label{LimitThirdTerm1Case2}
III_1 = -2 \Big( \partial_{\theta_1} \eta(|x^j|,\beta^j) + O \Big( \frac1{|x^j|} \Big)\Big) \to -2 \partial_{\theta_1} V(\widehat \theta)= -2 \partial_{\theta_1} V(\overline \theta) = -2(x_0)_1
\end{equation}
as $j\to\infty$. Plugging \eqref{LimitFirstTerm}, \eqref{LimitSecondTerm}, \eqref{LimitThirdTerm2}, \eqref{LimitThirdTerm1Case1} and \eqref{LimitThirdTerm1Case2} into \eqref{AsympExpansionU}, one gets our desired limit \eqref{LimitExpansionU}.

\noindent
\textbf{Part (2)}.
Suppose that there exists an unbounded sequence $x^j$ such that 
\[
\frac{\partial u}{\partial x_1}(x^j) < 0, \quad (x^j)_1 \geqslant \gamma_0+1\qquad \forall j \in \N.
\]
Thus, there exists a bounded sequence of positive numbers $\{a_j\}$ such that 
\[
u(x^j) > u(x_{a_j}), \quad x_{a_j} = x^j+ (2a_j,0,\dots,0) \qquad \forall j \in \N.
\]
Let define $\gamma_j =(x_j)_1+a_j > (x_j)_1$ to get that 
\begin{equation}\label{ContradictionExpansion}
\frac{|x^j|}{\gamma_j -(x_j)_1}  \Big( u(x^j)-u((x^j)^{\gamma_j}) \Big) >0.
\end{equation}
We split our next arguments in two cases. 

\textbf{Case 1.} $\liminf\limits_{j\to \infty} \gamma_j< \infty$.  Due to a passage to a subsequence, we can assume that $\gamma_j \to \gamma \geqslant \gamma_0+1$ as $j \to \infty$. Thus, \eqref{LimitExpansionU} implies that 
\[
\lim_{j\to\infty} \frac{|x^j|}{\gamma_j -(x^j)_1} \Big( u(x^j)-u((x^j)^{\gamma_j}) \Big) = -2L\gamma-2(x_0)_1 \leqslant -2L <0,
\]
a contradiction to \eqref{ContradictionExpansion}. 

\textbf{Case 2.} $\gamma = \infty$. We first observe that 
\[
\lim_{j\to\infty} \frac{|x^j|}{|(x^j)^{\gamma_j}|} =1
\]
and that $\gamma_j \leqslant |(x^j)^{\gamma_j}|$ since the definition of $\gamma_j$. Repeating the arguments above, we obtain that 
\begin{equation}\label{LimFirstTerm}
\frac{L|x^j|}{\gamma_j -(x^j)_1} \Big( |x^j|-|(x^j)^{\gamma}|  \Big) = -4L\gamma_j(1+o(1)),
\end{equation}
that
\begin{equation}\label{LimSecondTerm}
\frac{|x^j|}{\gamma_j-(x^j)_1} \Big( |x^j| \xi(|x^j|) - |(x^j)^{\gamma_j}| \xi( |(x^j)^{\gamma_j}|) \Big) =
\begin{cases}
O\Big( \frac{\gamma_j}{ |x^j|^{1-\varpi}} \Big) =O(\gamma_j^{\varpi}) \quad&\text{if } q\geqslant 4,\\
O\Big(\frac{\gamma_j}{|x^j|^{q-3}} \Big)=O(\gamma_j^{4-q}) \quad&\text{if } 3<q<4,
\end{cases}
\end{equation}
that
\begin{equation}\label{LimThirdTerm1}
  \frac{|x^j|}{\gamma_j-(x^j)_1} \Big( \eta(|x^j|,(\theta^j)^{\gamma_j}) - \eta( |(x^j)^{\gamma_j}|,(\theta^j)^{\gamma_j}) \Big) = O\Big( \frac{\gamma_j}{|x^j|} \Big)=O(1),
\end{equation}
and that
\begin{equation}\label{LimThirdTerm2}
\frac{|x^j|}{\gamma_j-(x^j)_1} \Big( \eta(|x^j|,(\theta^j)^{\gamma_j}) - \eta( |(x^j)^{\gamma_j}|,(\theta^j)^{\gamma_j}) \Big)= O\Big( \frac{\gamma_j}{|x^j|} \Big)=O(1).
\end{equation}
Combining \eqref{LimFirstTerm}, \eqref{LimSecondTerm}, \eqref{LimThirdTerm1} and \eqref{LimThirdTerm2} gives us 
\[
\frac{|x^j|}{\gamma_j-(x^j)_1} \Big( u(x^j)-u((x^j)^{\gamma_j}) \Big)= -4L\gamma_j(1+o(1))+ O(\gamma_j^{\varpi})+ O(1)+O(1) \to -\infty
\]
if $q\geqslant 4$ or 
\[
\frac{|x^j|}{\gamma_j-(x^j)_1} \Big( u(x^j)-u((x^j)^{\gamma_j}) \Big)= -4L\gamma_j(1+o(1))+ O(\gamma_j^{4-q})+ O(1)+O(1) \to -\infty
\]
if $3<q<4$. Those contradict \eqref{ContradictionExpansion} again. Thus, we conclude the proof of the lemma.
\end{proof}

\begin{proof}[\textbf{Proof of Theorem \ref{MainThm}.}]
We transform \eqref{MainEq} into the system of two second order elliptic equations
\begin{equation}\label{MainSystem}
\begin{cases}
-\Delta u= w \quad&\text{in }\R^3,\\
-\Delta w =-u^{-q} \quad&\text{in }\R^3.
\end{cases}
\end{equation}
Then, we establish the following lemma whose proof mimics that of  \cite[Lemma  4.2]{Troy81} and \cite[Lemma 5.3]{GWZ18}.
\begin{lemma}\label{LemTroy}
Let $\gamma \in \R$ and $(u,w)$ be a positive entire  solution of \eqref{MainSystem}. Suppose that 
\[
u(x) \leqslant u(x^{\gamma}), \quad u(x) \not\equiv u(x^{\gamma}) \quad\text{if } x_1<\gamma.
\]
Then, we claim that
\begin{equation}
u(x) < u(x^{\gamma}), \quad w(x) < w(x^{\gamma}) \quad\text{if } x_1<\gamma,
\end{equation}
and that
\begin{equation}
\frac{\partial u}{\partial x_1} >0, \quad \frac{\partial w}{\partial x_1} >0 \quad \text{on } \Gamma_{\gamma}.
\end{equation}
\end{lemma}

Now we are in position to prove Theorem \ref{MainThm}. To demonstrate the sufficiency of Theorem \ref{MainThm}, we show that there exists $\gamma'>0$ such that 
\begin{equation}\label{AssumeUW}
u(x) <u(x^{\gamma}), \quad w(x)<w(x^{\gamma}) \quad \text{for } \gamma \geqslant \gamma' \text{ and } (x)_1<\gamma.
\end{equation}
We assume that \eqref{AssumeUW} is false. By Lemma \ref{LemTroy}, there exists two sequences $\{\gamma^j\}_{j\geqslant 1} \subset  \R$ and $\{x^j\}_{j\geqslant 1} \subset \R^3$ such that $\lim\limits_{j\to \infty} \gamma^j = \infty$, $(x^j)_1 < \gamma^j$ and 
\[
u(x^j) \geqslant u(y^j), \quad y^j=(x^j)^{\gamma^j} \quad \text{for all } j.
\]
Notice that $\lim\limits_{j\to \infty} |y^j| = \infty$, which implies $\lim\limits_{j\to\infty} u(y^j)=\infty$. Hence, $|x^j|$ tends to infinity. From Lemma \ref{LemLimitExpansionU}, we obtain 
\[
(x^j)_1 \leqslant \gamma_0+1 = -\frac{(x_0)_1}L+1 
\]
for large $j$.  Hence, for any $\beta > \gamma_0+1$, there holds
\[
u(x^j) \geqslant u(y^j) \geqslant u((x^j)^{\beta})\quad \text{for large } j 
\]
since $((x^j)^{\gamma^j})_1 > ((x^j)^{\beta})_1$ for large $j$ and $u(x)$ tends to 0 as $|x|$ tends to infinity. We thus use Lemma \ref{LemTroy} again to get that 
\[
0\leqslant \frac{|x^j|}{\beta-(x^j)_1} \Big( u(x^j)- u((x^j)^{\beta}) \Big) = -2L\beta - 2(x_0)_1 <0.
\]
This contradiction shows us \eqref{AssumeUW}. The rest of the sufficient part is followed by that of \cite[Theorem 1.1]{Zou95} and \cite[Theorem 1.1]{GHZ15}, then we omit the details here. 

Now, we consider the necessary part of the proof. As pointed out in \cite[Theorem 1.6]{Gue12}, any minimal radial entire solution $u(r)$ of \eqref{MainEq} as $q>3$ satisfies 
\[ r^{-1}u(r) -L= 
\begin{cases}
O(r^{-1}) &\quad \text{if } q>4,\\
O(r^{-1}\log r) &\quad \text{if } q=4,\\
O(r^{3-q}) &\quad \text{if } 3<q<4,
\end{cases}
\]
for some positive $L$. This implies 
\[
r^{-1}u(r) - L= o(r^{-\vartheta})
\]
for $r$ large and $\vartheta \in(0,1)$. Thus, we conclude the necessity and the proof of Theorem \ref{MainThm}.
\end{proof}

\section*{Acknowledgments}

The author would like to thank Qu\^oc Anh Ng\^o for his discussion on the preparation of this paper.

\end{document}